\documentclass[11pt,colorinlistoftodos,oneside,reqno]{amsart}
\usepackage{minted}
\usepackage{cancel}
\usepackage[a4paper, total={6in, 9.7in}]{geometry}
\usepackage{subscript}
\usepackage{float}
\usepackage{amsfonts}
\usepackage{mathtools}
\usepackage{fancyhdr}
\mathtoolsset{showonlyrefs}
\numberwithin{equation}{section}
\usepackage{mathrsfs}  
\usepackage{amsthm}
\usepackage[final, pdftex, pdfpagelabels, pdfstartview = {FitH}, bookmarks, colorlinks, plainpages = false, linktoc=all, linkcolor=magenta, citecolor=cyan, urlcolor=blue,filecolor=black]{hyperref}
\usepackage{cleveref}
\usepackage{amssymb}
\usepackage[USenglish]{babel}
\usepackage{enumerate}
\usepackage{xcolor}
\usepackage{todonotes}
\usepackage{enumitem}
\usepackage{multirow, tabu}

\usepackage{dynkin-diagrams}
\usepackage{standalone}
\usepackage{comment}
\usepackage{tikz}
\usetikzlibrary{calc}
\usepackage{subcaption}
\usepackage[title]{appendix}
\usepackage{oldgerm}
\usepackage{multicol}
\usepackage[
backend=biber,
style=alphabetic,
sorting=nyt
]{biblatex}
\usepackage{todonotes}
\usepackage{xparse}
\usepackage{csquotes}

\setuptodonotes{size=\tiny}    

\NewDocumentCommand{\NewTodoAuthor}{mmm}{%
  \expandafter\NewDocumentCommand\csname #1\endcsname{s m}{%
    \IfBooleanTF{##1}
      {\todo[inline,author=#2,color=#3]{##2}}%
      {\todo[author=#2,color=#3]{##2}}%
  }%
}

\NewTodoAuthor{CR}{CR}{cyan!20}
\NewTodoAuthor{BM}{BM}{orange!25}
\NewTodoAuthor{DP}{DP}{blue!20}

\theoremstyle{plain}
\newtheorem{theorem}{Theorem}[section]
\newtheorem{lemma}[theorem]{Lemma}
\newtheorem{prop}[theorem]{Proposition}
\newtheorem{cor}[theorem]{Corollary}

\newtheorem{claim}[theorem]{Claim}

\theoremstyle{definition}
\newtheorem{definition}[theorem]{Definition}
\newtheorem{remark}[theorem]{Remark}
\newtheorem{example}[theorem]{Example}

\newtheorem{introtheorem}{Theorem}

\newcommand{\peso}[2]{{(#1,#2)_{\varpi}}}
\newcommand{\hh}[1]{\mathbf{H}_{#1}}
\newcommand{\HH}[1]{\underline{\mathbf{H}}_{#1}}

\newcommand{\newword}[1]{\emph{\textbf{#1}}}
\usepackage{amstext}
\usepackage{array}
\newcolumntype{L}{>{$}l<{$}} 
\newcolumntype{C}{>{$}c<{$}} 
\usepackage{dsfont}

\newcommand{\Ne}[1]{\mathbf{N}_{#1}}
\newcommand{\He}{\mathbf{H}}
\newcommand{\M}{\mathbf{M}}
\newcommand{\N}{\mathbf{N}}
\newcommand{\Z}{\mathbb{Z}}
\newcommand{\lemsthree}[3]{Lemmas~\ref{#1}, \ref{#2}, and~\ref{#3}}

\newcommand\pair[2]{\textcolor{#1}{\cancel{\textcolor{black}{#2}}}}

\newcommand{\pairA}[1]{\pair{red}{#1}}
\newcommand{\pairB}[1]{\pair{cyan}{#1}}
\newcommand{\pairC}[1]{\pair{teal}{#1}}
\newcommand{\pairD}[1]{\pair{violet}{#1}}

\title{Atomic decomposition for an affine Weyl group of type $G_2$}

\author{Bárbara Muniz}
\address{
Institute of Mathematics\\
    Jagiellonian University in Krak{\'o}w\\ 
    Poland
}
\email{baarbarasm@hotmail.com}

\author{David Plaza}
\address{
Instituto de Matem\'aticas \\
Universidad de Talca \\
Talca, Chile
}
\email{dplaza@utalca.cl}

\author{Claudia Rojas-Andías}
\address{
Facultad de Ciencias Básicas\\
Universidad Católica del Maule \\
Chile
}
\email{clrojas@ucm.cl}

\date{\today}

\newcommand\scalemath[2]{\scalebox{#1}{\mbox{\ensuremath{\displaystyle #2}}}}

\begin{document}

\begin{abstract}
We show that the elements of the Kazhdan--Lusztig basis of the spherical Hecke algebra of type $G_2$ have an atomic decomposition. 
As a by-product,  we obtain a new algorithm to compute generalized Kostka--Foulkes polynomials in type $G_2$. 
\end{abstract}

\maketitle

\section{Introduction}
Kazhdan–Lusztig polynomials arise as the transition coefficients between the canonical (Kazhdan–Lusztig) basis and the standard basis of the Hecke algebra attached to a Coxeter system \cite{kazhdan1979representations}. 
They are fundamental objects in representation theory and geometry because they encode important structural information in both settings.
In this paper we work with an affine Weyl group of type $G_2$ and focus on a distinguished family of  elements indexed by dominant weights, which we call \emph{spherical elements}.
They are  obtained by translating the longest element of the finite Weyl group by a dominant weight  (see Section~2 for the precise definition). 
For spherical elements,  Kazhdan–Lusztig polynomials coincide with  Kostka–Foulkes polynomials or $q$-weight multiplicities \cite{lusztig1983singularities,kato1982spherical}, and in particular their coefficients are known to be non-negative.

In type~A, Lascoux discovered a further positivity phenomenon for spherical elements \cite{lascoux1989cyclic}. 
He introduced a basis now called the atomic basis and proved that the canonical basis elements expand in this basis with coefficients in $\mathbb{N}[q]$. 
We refer to this phenomenon  by saying that Kazhdan--Lusztig basis elements have an atomic decomposition. 
Lecouvey and Lenart   \cite{lecouvey2021atomic} later studied atomic decompositions in other types using the theory of crystal.
Their work provides evidence that atomicity extends beyond type $A$ and naturally raises the problem of determining for which affine types, and for which dominant weights, the spherical elements admit atomic decompositions.

Beyond type $A$, atomicity may fail, and a precise characterization of the weights and affine types for which it holds is still out of reach. In this paper we solve this question for an affine Weyl group of type $G_2$.

For a dominant weight $\lambda$, we write $\HH{\lambda}$ for the spherical Kazhdan–Lusztig basis element and $\N_\lambda$ for the corresponding element of the atomic basis. Our main result is the following.

\begin{introtheorem}\label{thr:atomic}\it
For every dominant weight $\lambda$ we have an expansion
\[
\HH{\lambda}=\sum_{\mu\leq \lambda} a_{\lambda,\mu}(q)\N_\mu,
\]
where $\leq $ denotes the dominance order on weights and the polynomials $a_{\lambda,\mu}(q)$ have non-negative coefficients. In other words, in type $G_2$, all the Kazhdan--Lusztig  basis elements associated to spherical elements admits an atomic decomposition. 
\end{introtheorem}

We give two proofs of this result.

The first proof relies on the step-by-step decompositions of the pre-canonical bases introduced in \cite{libedinsky2022pre}. 
These decompositions do not immediately give atomicity, since one of the intermediate transitions is not positive. 
It is therefore necessary to verify that the negative terms that arise during the process cancel out, yielding the desired positive expansion.

The second proof is obtained by modifying the definition of the pre-canonical bases so that the step-by-step decompositions become positive at each stage, along the line of what was done in \cite{PatimoTorres2023} for type $C_2$. This approach is particularly well-suited for the construction of crystalline realizations of the atomic decomposition, an ongoing line of work by the first author. Such concrete realizations are extremely valuable, as they can be used to combinatorially describe Kostka–Foulkes polynomials as $q$-weight multiplicities — a long-standing open problem beyond type $A$.
We expect that similar modifications of the pre-canonical bases may recover step-by-step positivity in other types as well.

As an application of our results, we provide an explicit algorithm for computing generalized Kostka–Foulkes polynomials in type $G_2$.

The paper is organized as follows. Section~2 reviews the necessary background on the affine Weyl group of type $G_2$, the spherical Hecke algebra, the spherical elements, and the pre-canonical bases. Section~3 contains the first proof of \Cref{thr:atomic} via step-by-step decompositions. Section~4 develops the adjusted bases and presents the second proof. Section~5 describes the computation of generalized Kostka–Foulkes polynomials. Finally, we add an appendix with a \textsc{SageMath} implementation of our algorithms to compute the atomic decomposition, the pre-canonical bases and the Kostka--Foulkes polynomials. 

\section*{Acknowledgments}
This project began while the first two authors were participating in the semester program
\emph{Categorification and Computation in Algebraic Combinatorics}
at the Institute for Computational and Experimental Research in Mathematics (ICERM) in Providence, RI.
The program was supported by the National Science Foundation under Grant No.~DMS--1929284.
Bárbara Muniz was supported by Narodowe Centrum Nauki, grants 2021/43/D/
ST1/02290 and 2021/42/E/ST1/00162.
David Plaza was partially supported by FONDECYT--ANID Grant No.~1240199.
Claudia Rojas-Andías was partially supported by the National Doctoral Scholarship from the Agencia Nacional de Investigación y Desarrollo (ANID, Chile) through grant 21251181.
For the purpose of Open Access, the authors have applied a CC-BY public copyright
licence to any Author Accepted Manuscript (AAM) version arising from this submission.


\section{Preliminaries}
Throughout this paper we denote by $\mathbb{N}$ the set of non-negative integers.

\subsection{The affine Weyl group of type \texorpdfstring{$G_2$}{} }
In this paper we only work with a root system $\Phi$ of type $G_2$.
Therefore we fix an explicit realization of $\Phi$ inside $V=\mathbb{R}^2$ with the usual  inner product denoted by $\langle -,-\rangle$.

We take $\Delta=\{\alpha_{1},\alpha_2\}$ as the set of simple roots, where  
\begin{equation}
  \displaystyle   \alpha_{1}=(1,0) \quad \mbox{ and } \quad
  \alpha_{2}=\left(-\tfrac{3}{2},\,\tfrac{\sqrt{3}}{2}\right).  
\end{equation}
Then the root system $\Phi$ is given by
\begin{equation}
 \Phi=   \{ \pm \alpha_1, \pm\alpha_2,   \pm (\alpha_1+\alpha_2),   \pm (2\alpha_1+\alpha_2)  , \pm (3\alpha_1+\alpha_2), \pm (3\alpha_1+2\alpha_2)  \} . 
\end{equation}
The root system $\Phi$ is depicted in Figure \ref{fig: root G2}, where simple roots are given in red and the positive roots  that are not simple are drawn in blue. 
We denote by $\Phi^\vee =\{\alpha^{\vee} \mid \alpha\in \Phi\}$ the dual root system, where $\displaystyle \alpha^\vee= \frac{2\alpha}{\langle \alpha ,\alpha \rangle} $.

\begin{figure}[h!]
\vspace{-10pt}
\centering
\begin{tikzpicture}[scale=1.5, >=Latex, line cap=round, line join=round]
  \def\rS{1.2}                             
  \pgfmathsetmacro{\rL}{sqrt(3)*\rS}       
  \tikzset{
    lab/.style={font=\scriptsize, inner sep=0.6pt},  
  }

  \colorlet{posred}{red!75!black}
  \colorlet{posblue}{blue!75!black}
  \colorlet{negcol}{black}

  \fill (0,0) circle (1.2pt);

  \draw[posblue,->, thick] (0,0) -- (0,\rL)
    node[pos=1, lab, above] {$3\alpha_{1}+2\alpha_{2}$};
  \draw[posblue,->, thick] (0,0) -- ({\rL*cos(30)},{\rL*sin(30)})
    node[pos=1, lab, above] {$3\alpha_{1}+\alpha_{2}$};
  \draw[posred,->, thick] (0,0) -- ({\rL*cos(150)},{\rL*sin(150)})
    node[pos=1, lab, above] {$\alpha_{2}$};

  \draw[negcol,->, thick] (0,0) -- (0,-\rL)
    node[pos=1, lab, below] {$-3\alpha_{1}-2\alpha_{2}$};
  \draw[negcol,->, thick] (0,0) -- ({\rL*cos(210)},{\rL*sin(210)})
    node[pos=1, lab, below] {$-3\alpha_{1}-\alpha_{2}$};
  \draw[negcol,->, thick] (0,0) -- ({\rL*cos(330)},{\rL*sin(330)})
    node[pos=1, lab, below] {$-\alpha_{2}$};

  \draw[posred,->, thick] (0,0) -- (\rS,0)
    node[pos=1, lab, right] {$\alpha_{1}$};
  \draw[negcol,->, thick] (0,0) -- (-\rS,0)
    node[pos=1, lab, left] {$-\alpha_{1}$};

  \draw[posblue,->, thick] (0,0) -- ({\rS*cos(60)},{\rS*sin(60)})
    node[pos=1, lab, above] {$2\alpha_{1}+\alpha_{2}$};
  \draw[posblue,->, thick] (0,0) -- ({\rS*cos(120)},{\rS*sin(120)})
    node[pos=1, lab, above] {$\alpha_{1}+\alpha_{2}$};

  \draw[negcol,->, thick] (0,0) -- ({\rS*cos(240)},{\rS*sin(240)})
    node[pos=1, lab, below] {$-2\alpha_{1}-\alpha_{2}$};
  \draw[negcol,->, thick] (0,0) -- ({\rS*cos(300)},{\rS*sin(300)})
    node[pos=1, lab, below] {$-\alpha_{1}-\alpha_{2}$};
\end{tikzpicture}
\caption{A root system of type $G_2$}
\label{fig: root G2}
\end{figure}
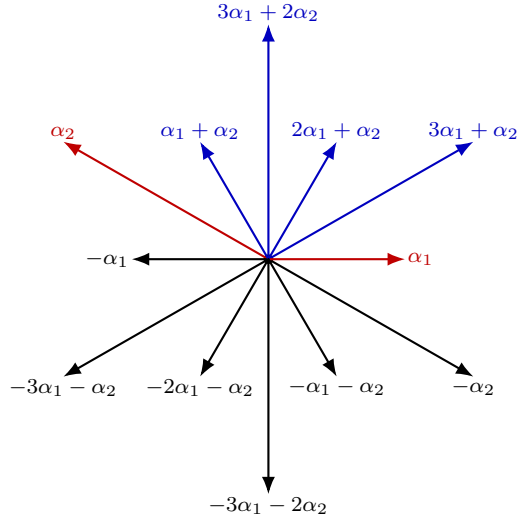

Because the definition of precanonical bases involves the height of positive roots, we include in Table \ref{tab:heights-positive} the height of each positive root.\\

\begin{table}[h!] 
\centering
\begin{tabular}{|c |c| c| c| c| c| c|}
\hline
\mbox{Root} &$\alpha_{1}$ & $\alpha_{2}$ & $\alpha_{1}+\alpha_{2}$ & $2\alpha_{1}+\alpha_{2}$ & $3\alpha_{1}+\alpha_{2}$ & $3\alpha_{1}+2\alpha_{2}$ \\
\hline
Height & $1$ & $1$ & $2$ & $3$ & $4$ & $5$ \\ 
\hline
       &  $ \peso{2}{-1}$ & $\peso{-3}{2}$  &  $\peso{-1}{1}$  & $\peso{1}{0}$ &  $\peso{3}{-1}$ & $\peso{0}{1}$.  \\
 \hline
\end{tabular}
\caption{Heights of the six positive roots in the $G_2$ root system.}
\label{tab:heights-positive}
\end{table}

The fundamental weights $\{\varpi_{1},\varpi_{2}\}$ are determined by the conditions 
$\langle \varpi_{i},\alpha_{j}^{\vee}\rangle=\delta_{i,j}$ for $i,j\in\{1,2\}$.
Concretely, we have $\varpi_1=2\alpha_1+\alpha_2$ and $\varpi_2= 3\alpha_1 +2\alpha_2$. 
Let $X=\mathbb{Z}\varpi_{1}\oplus\mathbb{Z}\varpi_{2}$ be the weight lattice and 
$X^{+}=\mathbb{N}\varpi_{1}\oplus\mathbb{N}\varpi_{2}$ the set of dominant weights. 
To avoid confusion with coordinates in $V$, for 
$\lambda=a_{1}\varpi_{1}+a_{2}\varpi_{2}\in X$ we use the shorthand $\peso{a_{1}}{a_{2}}$ to denote $\lambda$. 
We will frequently use the coefficients of the positive roots when written in the basis of fundamental weights, and we collect these coefficients in Table~\ref{tab:heights-positive}.
We also define $\rho \in X^+$ as the half-sum of positive roots. 
In formulas, 
\begin{equation}
    \rho = \frac{1}{2}\sum_{\alpha \in \Phi^{+}} \alpha  = 5\alpha_1 + 3\alpha_2 = \peso{1}{1}.
\end{equation}

For $\alpha \in \Phi$  and $m\in \mathbb{Z}$  we define the hyperplane $H_{\alpha,m}$ by 
\begin{equation}
    H_{\alpha , m }  = \{  \lambda \in V \mid \langle  \lambda , \alpha^\vee \rangle =m \}
\end{equation}
and the reflection $s_{\alpha, m}$ along the hyperplane $H_{\alpha , m }$. In formulas, we have $s_{\alpha, m}(\lambda) =  \lambda  - ( \langle\lambda , \alpha^\vee \rangle -m) \alpha$. 

The finite Weyl group  $W_f$ attached to $\Phi$ is the group generated by the reflections $\{ s_{\alpha, 0} \mid \alpha \in \Phi\} $. 
Let $s_1=s_{\alpha_{1},0}$ and $s_2=s_{\alpha_2 , 0}$.  It is well-known that $W_f$ is generated by $\{s_1,s_2\}$. 
Indeed $W_f$ is a Coxeter system with presentation
\begin{equation}
    W_f  = \langle   s_1,s_2 \mid s_1^2=s_2^2=1, \quad (s_1s_2)^6=1    \rangle. 
\end{equation}
So that $W_f$ is isomorphic to the  group of symmetries of  a regular  hexagon.  
It longest element is $w_0=(s_1s_2)^3=(s_2s_1)^3$. 

The affine  Weyl group  $W_f$ attached to $\Phi$ is the group generated by the reflections $\{ s_{\alpha, m} \mid \alpha \in \Phi, \,  m\in \mathbb{Z} \} $. 
The connected components of the complement of hyperplanes 
\begin{equation}
    V\setminus \bigcup_{\alpha, m } H_{\alpha , m }
\end{equation}
are called alcoves.  In our $G_2$-setting , alcoves correspond to triangles (see \Cref{fig:G2-alcove})
We denote by $\mathcal{A}$ the set of alcoves.  We call 
\begin{equation}
    A_0 = \{ \lambda \in V \mid  -1 < \langle \lambda , \alpha^\vee \rangle <0, \mbox{ for all } \alpha \in \Phi^+   \}
\end{equation}
the fundamental alcove. 
There is a bijection $W_a\rightarrow \mathcal{A}$ given by $w\mapsto w(A_0)$. 
This bijection allows us to uniquely identify elements in $W_a$ with alcoves. 

The walls of $A_0$ are supported in the hyperplanes $H_{\alpha_1,0}$, $H_{\alpha_{2},0}$ and $H_{\beta, -1}$, where $\beta $ is the highest short root, so that in our setting we have $\beta=2\alpha_1+\alpha_2$. Let $s_0=s_{\beta,-1}$. 
Then $W_a$ is also a Coxeter system with presentation
\begin{equation}
    W_f  = \langle s_0 ,  s_1,s_2 \mid s_0^2=s_1^2=s_2^2=1, \, (s_0s_1)^3= (s_0s_2)^2=(s_1s_2)^6=1   \rangle. 
\end{equation}
Also, the affine Weyl group $W_a$ can be realized as a semidirect product: $W_a=W_f \ltimes  \mathbb{Z}\Phi$, where $\mathbb{Z}\Phi$ acts on $V$ by translations. 
In type $G_2$ as the determinant of the Cartan matrix is $1$  we have $X= \mathbb{Z}\Phi$. 
Of course, this is not always the case and in general we  only have one containment $\mathbb{Z}\Phi \subset X$.  

\begin{definition} \label{def: theta}
    For $\lambda \in X^+$ let $t_{\lambda}$  be the translation by $\lambda$.  
    We define $\theta (\lambda) \in  W_a$ as $\theta (\lambda) = t_\lambda w_0$. 
\end{definition}

With the identification of $W_a$ with the set of alcoves aforementioned, it is easy to identify the alcove corresponding to $\theta (\lambda)$. 
Namely, find the alcove corresponding to $w_0$ and then translate this alcove by $\lambda$. 
The alcoves corresponding to $\theta (\lambda)$ for 
$$\lambda = \peso{0}{0}, \peso{1}{0}, \peso{2}{0}, \peso{3}{0}, \peso{0}{1}, \peso{1}{1}$$ 
are depicted in green in \Cref{fig:G2-alcove}, where the fundamental alcove is highlighted in blue.

\begin{figure}[htbp]
\centering
\begin{tikzpicture}[scale=1, line cap=round, line join=round]

\pgfmathsetmacro{\rt}{0.866025403784} 
\pgfmathsetmacro{\sq}{1.73205080757}  
\newcommand{\wt}[2]{({0.5*(#1)}, {\rt*(#1)+\sq*(#2)})}

\pgfmathsetmacro{\Lx}{3.5}
\pgfmathsetmacro{\Ly}{3.5}

\pgfmathsetmacro{\Amaxraw}{ceil( 2*\Lx )}
\pgfmathsetmacro{\Bpad}{\Ly/\sq}
\pgfmathsetmacro{\Bmaxraw}{ceil( \Bpad + \Amaxraw/2 )}
\pgfmathtruncatemacro{\Amax}{\Amaxraw}
\pgfmathtruncatemacro{\Bmax}{\Bmaxraw}

\newcommand{\drawFamilySmart}[2]{%
  \pgfmathsetmacro{\mminraw}{-(abs(#1)*\Amax + abs(#2)*\Bmax)}
  \pgfmathsetmacro{\mmaxraw}{ (abs(#1)*\Amax + abs(#2)*\Bmax)}
  \pgfmathtruncatemacro{\mmin}{\mminraw}
  \pgfmathtruncatemacro{\mmax}{\mmaxraw}
  \foreach \mm in {\mmin,...,\mmax} {%
    \ifnum #1=0
      \pgfmathsetmacro{\aBase}{0}
      \pgfmathsetmacro{\bBase}{\mm/(#2)}
    \else
      \pgfmathsetmacro{\aBase}{\mm/(#1)}
      \pgfmathsetmacro{\bBase}{0}
    \fi
    \pgfmathsetmacro{\dA}{#2}
    \pgfmathsetmacro{\dB}{-(#1)}
    \pgfmathsetmacro{\T}{\Amax + \Bmax + 2}
    \pgfmathsetmacro{\aS}{\aBase - \T*\dA}
    \pgfmathsetmacro{\bS}{\bBase - \T*\dB}
    \pgfmathsetmacro{\aE}{\aBase + \T*\dA}
    \pgfmathsetmacro{\bE}{\bBase + \T*\dB}
    \draw[line width=0.45pt] \wt{\aS}{\bS} -- \wt{\aE}{\bE};
  }%
}

\newcommand{\rootToAB}[2]{\draw[red!80!black, line width=1.0pt, ->] (0,0) -- \wt{#1}{#2};}

\begin{scope}
  \clip (-\Lx,-\Ly) rectangle (\Lx,\Ly);

  \fill[orange!75, opacity=0.35] (0,0) -- \wt{\Amax}{0} -- \wt{0}{\Bmax} -- cycle;

  \fill[green!85!black, opacity=0.60, draw=yellow!60!black, line width=1pt]
    (0,0) -- \wt{0.5}{0} -- \wt{0}{1/3} -- cycle;
  \fill[green!85!black, opacity=0.60, draw=yellow!60!black, line width=1pt]
    \wt{1}{0} -- \wt{1.5}{0} -- \wt{1}{1/3} -- cycle;
  \fill[green!85!black, opacity=0.60, draw=yellow!60!black, line width=1pt]
    \wt{0}{1} -- \wt{1/2}{1} -- \wt{0}{4/3} -- cycle;
  \fill[green!85!black, opacity=0.60, draw=yellow!60!black, line width=1pt]
    \wt{2}{0} -- \wt{5/2}{0} -- \wt{2}{1/3} -- cycle;
  \fill[green!85!black, opacity=0.60, draw=yellow!60!black, line width=1pt]
    \wt{3}{0} -- \wt{7/2}{0} -- \wt{3}{1/3} -- cycle;

  \fill[green!85!black, opacity=0.60, draw=yellow!60!black, line width=1pt]
    \wt{1}{1} -- \wt{3/2}{1} -- \wt{1}{4/3} -- cycle;

  \fill[cyan!70, opacity=0.55, draw=cyan!60!black, line width=1pt]
    (0,0) -- \wt{-0.5}{0} -- \wt{0}{-1/3} -- cycle;

  \drawFamilySmart{1}{0}   
  \drawFamilySmart{0}{1}   
  \drawFamilySmart{2}{3}   
  \drawFamilySmart{1}{3}   
  \drawFamilySmart{1}{1}   
  \drawFamilySmart{1}{2}   

  \foreach \bInd in {-\Bmax,...,\Bmax}{
    \foreach \aInd in {-\Amax,...,\Amax}{
      \fill \wt{\aInd}{\bInd} circle (1.8pt);
    }
  }
\end{scope}

\rootToAB{ 2}{-1}  
\rootToAB{ 1}{ 0}  
\rootToAB{-1}{ 1}  
\rootToAB{-2}{ 1}  
\rootToAB{-1}{ 0}  
\rootToAB{ 1}{-1}  

\rootToAB{ 3}{-1}  
\rootToAB{ 0}{ 1}  
\rootToAB{-3}{ 2}  
\rootToAB{-3}{ 1}  
\rootToAB{ 0}{-1}  
\rootToAB{ 3}{-2}  

\fill (0,0) circle (2.0pt);

\end{tikzpicture}

\caption{Type \(G_2\) alcoves. The dots correspond to $X$. The dominant chamber is  highlighted. The fundamental alcove is colored in light blue. The green triangle correspond to $\theta$-elements.}
\label{fig:G2-alcove}
\end{figure}
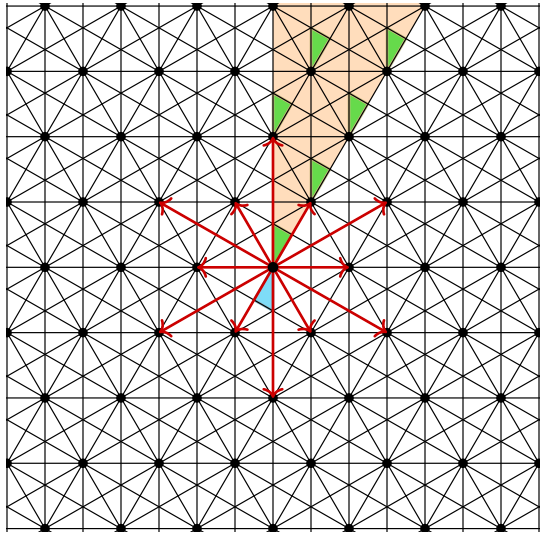

\subsection{Pre-canonical bases}

To any Coxeter system we can attach a Hecke algebra. 
Let $\mathcal{H}$ be the Hecke algebra associated to $W_a$. 
It is the  $\mathbb{Z}[v,v^{-1}]$-algebra associative algebra with generators $\{ \He_0,\He_1, \He_2 \}$ and relations 
\begin{equation}
    \He_i^2=(v-v^{-1})\He_i +1, \quad \He_0\He_1\He_0=\He_1\He_0\He_1, \quad \He_0\He_2=\He_2\He_0, \quad  (\He_1\He_2)^3    =(\He_2 \He_1)^3.
\end{equation}

It comes equipped with two distinguished bases: 
The standard basis $\{ \hh{w} \mid w \in W_a \}$ and the canonical or Kazhdan--Lusztig basis $\{\HH{w} \mid w\in W_a\}$ (see for instance \cite{elias2020introduction}).
The coefficients of the change of basis matrix between the canonical and standard basis are the Kazhdan--Lusztig polynomials. 
More concretely, the Kazhdan--Lusztig polynomials, $h_{x,y}(v)\in \mathbb{Z}[v,v^{-1}]$, are given by 
\begin{equation}
    \HH{w} = \sum_{x\in W_a} h_{x,w}(v) \hh{x}.
\end{equation}
Indeed, it is known that $h_{x,w}(v)\in \mathbb{N}[v]$ for arbitrary Coxeter systems \cite{EliasWilliamson2014}. 
We have a third basis $\{  \Ne{w} \mid w \in W_a\} $ of $\mathcal{H}$, which we refer as the atomic basis. 
Their elements are given by 
\begin{equation}
    \Ne{w} = \sum_{x\leq w} v^{\ell(w) -\ell (x) } \hh{x}, 
\end{equation}
where $\leq $ denotes the Bruhat order and $\ell(\cdot)$ the length function. 

\begin{definition} \label{def:atomicity}
    Let $\He \in \mathcal{H}$ and write it in terms of the atomic basis:
    \begin{equation}
        \He =\sum_{x\in W_a} a_{x}(v) \Ne{x}
    \end{equation}
    for some  $a_{x}(v) \in \mathbb{Z}[v,v^{-1}]$. 
    Then, we say that $\He$ has an atomic decomposition if  $a_{x}(v)\in \mathbb{N}[v]$ for all $x\in W_a$.
\end{definition}

\begin{remark}
It is known that, for any Coxeter system, the Kazhdan--Lusztig polynomials satisfy the \emph{monotonicity property} (see, \cite{plaza2017graded}):
\begin{equation}
  h_{x,w}(v) - v^{\,\ell(y)-\ell(x)}\, h_{y,w}(v) \in \mathbb{N}[v]
\qquad\text{for all } x \le y \le w .  
\end{equation}
Having an atomic decomposition is a substantially stronger condition than monotonicity. 
However, it is also known that many elements do not admit an atomic decomposition. 
In fact, the existence of an atomic decomposition appears to be a rather rare phenomenon.
\end{remark}

For $\lambda \in  X^{+}$ we set $\HH{\lambda} =\HH{\theta (\lambda)}$ and $\Ne{\lambda} =\Ne{\theta (\lambda)}$. 
The spherical Hecke algebra $\mathcal{H}^{\textbf{sph} } $ lives inside $\mathcal{H}$. 
It is the free $\mathbb{Z}[v,v^{-1}]$-submodule of $\mathcal{H}  $ with canonical basis $\{ \HH{\lambda}  \mid \lambda \in X^{+} \} $ and atomic basis $ \{ \N_{\lambda} \mid \lambda \in X^+  \}$. 
Although $\mathcal{H}^{\textbf{sph} } \subset \mathcal{H}$ it is not a sub-algebra of $\mathcal{H}$ since the multiplication in $\mathcal{H}^{\textbf{sph}}$ is defined as a  deformation of the usual multiplication in $\mathcal{H}$. 
In this paper we only use the linear structure of $\mathcal{H}^{\textbf{sph} } $ , so that to avoid unnecessary technicalities we left the multiplication in $\mathcal{H}^{\textbf{sph} } $ undefined and refer the interested reader to \cite[Section 2.2]{libedinsky2022pre}.

We write 
\begin{equation}
    \HH{\lambda} = \sum _{\mu \leq  \lambda }a_{\lambda , \mu} (q)  \Ne{\mu}, 
\end{equation}
where $q=v^2$ and $\leq $ denotes the dominance order on $X$. 
In type $A$ the polynomials $a_{\lambda , \mu}(q)$ are the atomic polynomials introduced by Lascoux  \cite{lascoux1989cyclic}. 
The main result in this paper (\Cref{thr:atomic}) is to prove that $a_{\lambda,\mu}(q)$ have positive coefficients or, following  Definition \ref{def:atomicity}, that the elements of the canonical basis have an atomic decomposition.

\medskip
We now define pre-canonical bases and explain their connection with the atomic decomposition. 
First, we need to give some notation. 

\begin{definition}
    Let $\lambda \in X$. 
    \begin{enumerate}
        \item  We say that $\lambda $ is singular if there is an element $w\in W_f$ which fixes $\lambda + \rho$. 
    A non-singular weight is called regular. 
    \item The dot action of $W_f$ in $V$ is given by the formula:  $w\cdot \lambda = w(\lambda +\rho)-\rho$.  
    \item If $\lambda \in X$ is regular then we define $w_{\lambda} \in W_f$ to be the unique element such that $w_\lambda \cdot \lambda \in X^{+}$. 
    \item For $\lambda \in X$ we define 
\begin{equation}
    \tilde{\HH{}}_{w} = \begin{cases}
        (-1)^{\ell (w_{\lambda})} \HH{w_\lambda \cdot \lambda},  & \mbox{if } \lambda \mbox{ is regular}\\[4pt]
        0, &  \mbox{if } \lambda \mbox{ is singular.}
    \end{cases}
\end{equation}
\item For $i\geq 1$ we define 
\begin{equation}
    \Phi^{\geq i } = \{ \alpha \in \Phi^+ \mid \operatorname{ht}(\alpha) \geq i \}. 
\end{equation}
\item For $I\subset \Phi$ we set $\Sigma_I = \sum_{\alpha \in I } \alpha$.
    \end{enumerate} 
\end{definition}

We are now in position to define pre-canonical bases. 

\begin{definition}
    For $i \geq 2$ and $\lambda \in X^+$ we define
    \begin{equation}
        \Ne{\lambda}^i = \sum _{I\subset \Phi^{\geq i}} (-q)^{|I|} \tilde{\HH{}}_{\lambda- \Sigma_I}.
    \end{equation}
    We call  $\mathcal{N}^i = \{ \Ne{\lambda } \mid \lambda \in X^{+} \} $ the $i$-th pre-canonical basis  of $\mathcal{H}^{\textbf{sph}}$. 
 \end{definition}

The motivation behind this definition is the anti-atomic formula \cite[Theorem 3.1]{libedinsky2022pre} which  asserts that $\mathcal{N}^2$ coincides with the atomic basis.  This is, $\Ne{\lambda} = \Ne{\lambda}^2$ for all $\lambda \in X^+$. 
It is worth to note that the anti-atomic formula hold for arbitrary affine Weyl groups. 
Retuning to the $G_2$-setting, it follows from \Cref{tab:heights-positive} that $\Phi^{\geq m} = \emptyset $ for all $m\geq 6$. 
Thus, $\Ne{\lambda}^m = \HH{\lambda}$ for all $m\geq 6$ and $\lambda \in X^+$. 
Therefore, in order to compute $\HH{\lambda}$ in terms of the atomic basis it is enough to compute the expansion of  an element $\Ne{\lambda}^{i+1}$ in terms of $\mathcal{N}^i$ for $2 \leq i \leq 5$. 

\begin{remark}

In type $A$, it was conjectured in \cite[Conjecture 1.6]{libedinsky2022pre} that, for every $i \ge 2$, the coefficients appearing in the expansion of an element $\Ne{\lambda}^{i+1}$ in terms of the basis $\mathcal{N}^i$ are polynomials with nonnegative coefficients.
This conjecture was recently proved by Yamil Sagurie and the first author \cite{PlazaSagurie25}, using methods that are closely related to the ones developed in this paper.
We emphasize that this step-by-step positivity property implies the existence of an atomic decomposition.

In type $G_2$, however, this step-by-step positivity does \emph{not} hold in general: as we will see in the next section, negative coefficients may appear in some of the intermediate expansions.
Nevertheless---and this is one of the striking features of type $G_2$---these negative terms ultimately cancel out when all steps are combined.
In other words, although the intermediate layers of the process fail to be positive, the \emph{total} expansion remains positive, and the atomic decomposition holds in type $G_2$.
 
\end{remark}

\section{First approach}

\subsection{Inverse step-by-step decompositions}

We begin this section by introducing certain elements that relax the definition of pre-canonical bases elements.

\begin{definition}
    Let  $A\subset \Phi$ and $\lambda \in X$. 
    We define 
   \begin{equation}
        \M^{A}_{\lambda}=\displaystyle \sum_{I \subset A} (-q)^{|I|} \mathbf{\underline{\widetilde{H}}}_{\lambda-\sum_{I}} .
    \end{equation}
\end{definition}

In the following lemma we collect our two main tools to obtain inverse step-by-step decompositions.

\begin{lemma}\label{lemma -1}
For all $A \subseteq \Phi$ and  $\lambda \in X$ we have
\begin{enumerate}[label=(\alph*),ref=(\alph*)]
    \item \label{item A} If  $\alpha \in A$ and  $B=A\setminus\{ \alpha \}$ then $ \M^{A}_{\lambda} = \M^{B}_{\lambda} -q \M^{B}_{\lambda -\alpha}$.
    \item  \label{item B}
        $\M^{A}_\lambda=-\M^{s_{i}(A)}_{s_{i}\cdot \lambda}$. In particular, if  $\langle \lambda , \alpha_{i}^\vee\rangle  =-1$ and  $s_{i}(A)=A$ then $\M^{A}_{\lambda}=0$. 
\end{enumerate}
\end{lemma}

\begin{proof}
   The first claim follows directly by the definition of $\M$-elements. 
   The second one is the content of \cite[Proposition 4.3]{libedinsky2022pre}. 
\end{proof}

We now obtain all the inverse step-by-step decompositions. 

\begin{prop} \label{prop: inverse decomp}
    Let $\lambda = \peso{\lambda_1}{\lambda_2} \in X^+$. 
    We have
\begin{equation}
\begin{array}{lll}
       \N_{\lambda}^5  = \scalemath{0.85}{\left\{   \begin{array}{rl}
         \N_{\lambda}^6  -q \N_{\peso{\lambda_1}{\lambda_2-1}}^6,    &  \mbox{if } \lambda_2\geq 1; \\
         &  \\
          \N_{\lambda}^6 ,   & \mbox{if } \lambda_2=0.
       \end{array}   \right.}    &    \quad  
 \N_{\lambda}^4  = \scalemath{0.85}{\left\{   \begin{array}{rl}
         \N_{\lambda}^5  -q \N_{\peso{\lambda_1-3}{\lambda_2+1}}^5,    &  \mbox{if } \lambda_1\geq 3; \\
         &  \\
          \N_{\lambda}^5 ,   & \mbox{if } \lambda_1=2;\\
          &  \\
          \N_{\lambda}^5 +q \N^{5}_{\peso{0}{\lambda_2}},   & \mbox{if } \lambda_1=1;\\
          & \\
          \N^{5}_{\lambda} + q \N^{5}_{\peso{1}{\lambda_2 -1}}, & \mbox{if } \lambda_1=0, \lambda_2 \geq 1;\\
          & \\
          \N^{5}_{\lambda} & \mbox{if } \lambda_1=\lambda_2=0.
       \end{array}   \right. }    \\[20pt]
          &    \\
  \N_{\lambda}^3  = \scalemath{0.85}{\left\{   \begin{array}{rl}
         \N_{\lambda}^4  -q \N_{\peso{\lambda_1-1}{\lambda_2}}^4,    &  \mbox{if } \lambda_1\geq 1; \\
         &  \\
          \N_{\lambda}^4 - q^{2} \N^{4}_{\peso{2}{\lambda_2 -2}} ,   & \mbox{if } \lambda_1=0, \lambda_2 \geq 2;\\
         & \\
         \N^{4}_{\lambda}, & {\text{else}}
       \end{array}   \right. }   & \quad 
      \N_{\lambda}^2  = \scalemath{0.85}{\left\{   \begin{array}{rl}
         \N_{\lambda}^3  -q \N_{\peso{\lambda_1+1}{\lambda_2-1}}^3,    &  \mbox{if } \lambda_2\geq 1; \\
         &  \\
         \N^{3}_{\lambda}, & \mbox{if } \lambda_2=0.
       \end{array}   \right. }
    \end{array}
\end{equation}
\end{prop}

\begin{proof}
All the decompositions follow by combining both claims in Lemma \ref{lemma -1}. 
We only prove the decomposition of $\N_{\lambda}^{4}$ in $\mathcal{N}^5$, since this is the most interesting case. 
The other cases are dealt with similarity and are left to the reader.

We recall from Table \ref{tab:heights-positive} that $\Phi^{\geq 5} = \Phi^{\geq 4} \cup \{ \peso{3}{-1}\}$ and that $s_1(\Phi^{\geq 5}) =\Phi^{\geq 5}$.

\begin{itemize}
    \item If $\lambda_1 \geq 3$ then $\peso{\lambda_1-3}{\lambda_2+1} \in X^+$ and the result follows from Lemma \ref{lemma -1} \ref{item A}. 
    \item If $\lambda_1 =2$ then $\N_{\lambda}^4 = \N_{\lambda}^5- q\N_{\peso{-1}{\lambda_2+1}}^5$. By Lemma \ref{lemma -1} \ref{item B} applied to $s_1$ we obtain $\N_{\peso{-1}{\lambda_2+1}}^5=0$. 
    Thus, $\N_{\lambda}^4 = \N_{\lambda}^5$ as we wanted to show. 
    \item If $\lambda_1 =1$ then $\N_{\lambda}^4 = \N_{\lambda}^5- q\N_{\peso{-2}{\lambda_2+1}}^5$. By Lemma \ref{lemma -1} \ref{item B} applied to $s_1$ we obtain $\N_{\peso{-2}{\lambda_2+1}}^5=-\N_{\peso{0}{\lambda_2}}^5$. 
    Therefore, $\N^{4}_{\lambda}=\N^{5}_{\lambda} + q \N^{5}_{\peso{0}{\lambda_2}}$.
    \item If $\lambda_1=0$ then $\N^{4}_{\lambda}=\N^{5}_{\lambda}-q \N^{5}_{\peso{-3}{\lambda_2+1}}$. By Lemma \ref{lemma -1} \ref{item B} applied to $s_1$ we obtain 
    $\N^{5}_{\peso{-3}{\lambda_2+1}}= -\N_{\peso{1}{\lambda_2-1}}^5$. If $\lambda_2\geq 1$ then $\peso{1}{\lambda_2-1} \in X^+$ and  we are done. 
    Otherwise, $\lambda_2=0$ and $\N_{\peso{1}{-1}}^5 = \tilde{\HH{}}_{\peso{1}{-1}} -q \tilde{\HH{}}_{\peso{1}{-2}}  $. Since both   $\peso{1}{-1}$ and $ \peso{1}{-2}$ are singular weights, we conclude that $\N_{\peso{1}{-1}}^5=0$ and the result follows.    
\end{itemize}
\end{proof}

\subsection{Step-by-step decompositions}
\label{sec:SBS-decompositions}
In this section we invert the decompositions obtained in \Cref{prop: inverse decomp}. 
Among these, the decompositions of $\mathcal{N}^{5}$ inside $\mathcal{N}^{6}$ and of $\mathcal{N}^{2}$ inside $\mathcal{N}^{3}$ are the simplest to analyze. 
Indeed, in both situations the rule is the same: we subtract the corresponding root when the resulting weight remains dominant; if it is no longer dominant, the decomposition consists of a single term. 
This simplicity is mirrored in the inverse decompositions: they are obtained by iteratively subtracting the relevant root, summing over all dominant weights encountered, until a non-dominant weight is reached. 
We record these two cases in the following lemma.

\begin{lemma} \label{lem: N6 in N5 and N3 in N2}
    Let $\lambda = \peso{\lambda_1}{\lambda_2} \in X^{+}$. Then, we have 
    \begin{equation}
             \begin{array}{ccc}
  \displaystyle    \N_{\lambda}^{6} = \sum_{i=0}^{\lambda_2} q^i \N_{\peso{\lambda_1}{\lambda_2-i}}^5     & \mbox{and } & \displaystyle  \N_{\lambda }^3  = \sum_{i=0}^{\lambda_2} q^i \N_{\peso{\lambda_1+i}{\lambda_2-i}}^2. \\            
        \end{array}
    \end{equation}
 In particular, both expressions are positive.    
\end{lemma}

\begin{proof}
Both identities follow by an easy inductive argument  on $\lambda_{2}$ together with Proposition \ref{prop: inverse decomp}.  
\end{proof}

We now move our attention to the decomposition of elements of $\mathcal{N}^{5}$ in terms of $\mathcal{N}^4$.

\begin{lemma}\label{lemma N5 enN4}
       Let $\lambda = \peso{\lambda_1}{\lambda_2} \in X^{+}$ and write $\lambda_1= 3m+r$ for $m\geq 0$ and $r\in \{0,1,2\}$. 
       Then, we have

       \begin{equation} \label{eq: N5 en N4}
           \N_{\lambda}^5 =  \left\{
           \begin{array}{ll}
     \displaystyle       \sum_{i=0}^m  q^i \N_{\peso{\lambda_1-3i}{\lambda_2+i}}^4   -q^{m}  \sum_{i=1}^{m+\lambda_2} q^{2i-1}  \N_{\peso{1}{m+\lambda_2-i}}^3 , &  \mbox{if } r=0; \\[10pt]
  \displaystyle       \sum_{i=0}^{m-1} q^i \N_{\peso{\lambda_1-3i}{\lambda_2+i}}^4   +q^{m}  \sum_{i=0}^{m+\lambda_2}   q^{2i}\N_{\peso{1}{m+\lambda_2-i}}^3 , &  \mbox{if } r=1; \\[10pt]
   \displaystyle       \sum_{i=0}^m  q^i \N_{\peso{\lambda_1-3i}{\lambda_2+i}}^4  , &  \mbox{if } r=2, 
           \end{array}  \right.
       \end{equation}
       where we make the convention that empty sums are set equal to $0$. 
\end{lemma}

\begin{proof}
  Let     $ \hat{\N}^5_{\lambda} $ be the right-hand side of \eqref{eq: N5 en N4}. 
  We must show that $\hat{\N_{\lambda}^5}= \N_{\lambda}^5$ for all $\lambda \in X^+$. 
  We proceed by induction on $\lambda_1$. 
  The base cases being $\lambda_1\in \{0,1,2\}$.

A direct and easy computation shows  that $\hat{\N_{\lambda}^5}= \N_{\lambda}^5$ for $\lambda = \peso{0}{0}, \peso{1}{0}, \peso{2}{0}$. 
So that we fix $\lambda_2 \geq 1$ and assume that  $\hat{\N_{\lambda}^5}= \N_{\lambda}^5$ for all $\mu =\peso{\mu_1}{\mu_2} $ with $\mu_1\in \{0,1,2\}$ and $\mu_2<\lambda_2$. 
By definition of the $\hat{\N}$-elements we have
\begin{equation} \label{eq: N5 en N4 base Induction}
\hat{\N}_{\lambda}^5 = \left\{   \begin{array}{ll}
    \N_{\lambda}^4 - q\hat{\N}_{\peso{1}{\lambda_2-1} }^5 & \mbox{if } \lambda_1=0; \\
 \N_{\lambda}^3 +q^2\hat{\N}_{\peso{1}{\lambda_2-1} }^5     & \mbox{if } \lambda_1=1; \\
   \N_{\lambda}^4  & 
   \mbox{if } \lambda_1=2. \\
\end{array} \right. 
\end{equation}

Now using our inductive hypothesis we have $\hat{\N}_{\peso{1}{\lambda_2-1} }^5 =\N_{\peso{1}{\lambda_2-1} }^5$. 
Then, using Proposition \ref{prop: inverse decomp} to write $\N_{\lambda}^3$ and $\N_{\lambda}^4$ in terms of $\mathcal{N}^5$, we see that in the three cases the right-hand side of \eqref{eq: N5 en N4 base Induction} reduces to $\N_{\lambda}^5$. 
This completes the proof of our base cases. 

We now fix  $\lambda_1\geq 3$ and suppose that $\hat{\N_{\mu}^5}= \N_{\mu}^5$ for all $\mu=\peso{\mu_1}{\mu_2}$  such that $\mu_1<\lambda_1$.  
We have
\begin{equation}
    \hat{\N}_{\lambda}^5 = \N_{\lambda}^4+q\hat{\N}_{\peso{\lambda_1-3}{\lambda_2+1}}^5=\N_{\lambda}^4+q\N_{\peso{\lambda_1-3}{\lambda_2+1}}^5 =\N_{\lambda}^5, 
\end{equation}
where the first equality comes from the definition of the $\hat{\N}$-elements, the second by our inductive hypothesis and the third by Proposition \ref{prop: inverse decomp}.
This finishes our induction and the proof of the lemma. 
\end{proof}

\begin{remark}
The alert reader may have noticed the appearance of ${\N}^3$–elements in the expansions of Lemma \ref{lemma N5 enN4}. 
Although one might expect only $\N^4$–terms, including $\N^3$-term makes the expressions more compact and will be essential in the next section, where we prove that the canonical basis elements admit an atomic decomposition.
If one wishes to express an $\N^5$–element solely in terms of the $\mathcal{N}^4$–basis, it is enough to combine Proposition \ref{prop: inverse decomp} with Lemma \ref{lemma N5 enN4}. 
In particular, such an expansion generally involves negative coefficients.     
\end{remark}

Lastly, we describe the expansion of $\mathcal{N}^4$ into $\mathcal{N}^3$.

\begin{lemma} \label{lemma N4 enN3}
  Let $\lambda = \peso{\lambda_1}{\lambda_2}\in X^+$ and write $\lambda_2=2m+r$ with $m\geq 0$ and $r\in \{0,1\}$. 
  Then, we have 
\begin{equation}\label{eq: N4 en N3}
  \displaystyle   \N_{\lambda}^4 = 
  \sum_{i=0}^{\lambda_1}  q^i \N^3_{\peso{\lambda_1-i}{\lambda_2}}+ q^{\lambda_1}\sum_{i=1}^{m} q^{4i-2} \left( \N^3_{\peso{2}{\lambda_2-2i}}+q \N^3_{\peso{1}{\lambda_2-2i}}+q^2 \N^3_{\peso{0}{\lambda_2-2i}} \right).
\end{equation}
\end{lemma}

\begin{proof}
    Let $\hat{\N}^4_{\lambda}$ be the right-hand side of \eqref{eq: N4 en N3}. 
    We must show that $\hat{\N}_{\lambda}^4 =\N_{\lambda}^4  $ for all $\lambda =\peso{\lambda_1}{\lambda_2}\in X^+$. 
    We proceed by induction on $\lambda_2$. 
    Let us further assume that $\lambda_2$ is even.
    Our base case is $\lambda_2=0$. 
    Notice that in this case the second row of the expression of $\N_{\lambda}^3$ in $\mathcal{N}^4$ in Proposition \ref{prop: inverse decomp} is never required. 
    On the other hand, since $m=0$ in this case, the element $\hat{\N}_\lambda^4$ reduces to 
    \begin{equation}
        \hat{\N}_\lambda^4 = \sum_{i=0}^{\lambda_1}  q^i \N^3_{\peso{\lambda_1-i}{\lambda_2}}.
    \end{equation}
    Therefore, we are essentially in the same situation of Lemma \ref{lem: N6 in N5 and N3 in N2}.
    Thus, an easy inductive argument on $\lambda_1$ shows that $\hat{\N}_{\lambda}^4 =\N_{\lambda}^4  $. 

    We now assume that $\lambda_2 \geq 2$ and that the result holds for all even numbers less than $\lambda_2$. 
    Let us suppose that $\lambda_1=0$. 
    By the definition of the $\hat{\N}^4$-elements, our inductive hypothesis and Proposition \ref{prop: inverse decomp} we obtain
    \begin{equation}
       \hat{\N}_\lambda^4 = \N_\lambda^3 +q^2   \hat{\N}_{\peso{2}{\lambda_2-2}}^4 = \N_\lambda^3 +q^2   {\N}_{\peso{2}{\lambda_2-2}}^4 =  {\N}_\lambda^4, 
    \end{equation}
    as we wanted to show. 
    
    Similarly, for $\lambda_1>0$ we have
    \begin{equation}
       \hat{\N}_\lambda^4 = \N_\lambda^3 +q   \hat{\N}_{\peso{\lambda_1-1}{\lambda_2}}^4 = \N_\lambda^3 +q^2   {\N}_{\peso{2}{\lambda_2-2}}^4 =  {\N}_\lambda^4. 
    \end{equation}
    This completes the proof of $\hat{\N}_{\lambda}^4 =\N_{\lambda}^4  $ when $\lambda_2$ is even. 
    The case when $\lambda_2$ is odd is entirely analogous and for this reason we omit it.    
\end{proof}

\subsection{Atomic Decomposition}

In this section we prove that the elements of the canonical basis $\HH{\lambda} =\N_{\lambda}^{6}$ have an atomic decomposition for all $\lambda\in X^{+}$. 

\begin{proof}[Proof of \Cref{thr:atomic}]
By \lemsthree{lem: N6 in N5 and N3 in N2}{lemma N5 enN4}{lemma N4 enN3}, it suffices to show that 
$\N^5_{\lambda}$ has an atomic decomposition whenever $\lambda_1\equiv 0 \pmod{3}$.
We first treat the case $\lambda_1=0$. 
If $\lambda_2=0$ then $\N^5_{\lambda}=\N^4_{\lambda}$ and there is nothing to prove. 
Thus we assume $\lambda_2\ge1$.  
By Lemma \ref{lemma N5 enN4} we have
\begin{equation}\label{eq:atomica1}
\N^5_{\peso{0}{\lambda_2}}
=
\N^4_{\peso{0}{\lambda_2}}
-
\sum_{i=1}^{\lambda_2} q^{2i-1}\,\N^3_{\peso{1}{\lambda_2-i}}.
\end{equation}

We now distinguish the parity of $\lambda_2$.
Assume  that $\lambda_2=2m$ is even.  
(The odd case is analogous and omitted.)
Using $\lambda_2=2m$, equation~\eqref{eq:atomica1} becomes
\begin{equation}\label{eq:atomic2}
\N^5_{\peso{0}{\lambda_2}}
=
\N^4_{\peso{0}{\lambda_2}}
-
\sum_{i=1}^{m} q^{4i-1}\,\N^3_{\peso{1}{\lambda_2-2i}}
-
\sum_{i=1}^{m} q^{4i-3}\,\N^3_{\peso{1}{\lambda_2-2i+1}}.
\end{equation}

On the other hand, Lemma~\ref{lemma N4 enN3} yields
\begin{equation}\label{eq:atomic3}
\N^4_{\peso{0}{\lambda_2}}
=
\N^3_{\peso{0}{\lambda_2}}
+
\sum_{i=1}^{m} q^{4i-2}
\Bigl(
\N^3_{\peso{2}{\lambda_2-2i}}
+
q\,\N^3_{\peso{1}{\lambda_2-2i}}
+
q^2\,\N^3_{\peso{0}{\lambda_2-2i}}
\Bigr).
\end{equation}

Substituting \eqref{eq:atomic3} into \eqref{eq:atomic2} and regrouping terms gives
\begin{equation}\label{eq:atomic4}
\N^5_{\peso{0}{\lambda_2}}
=
\sum_{i=0}^{m} q^{4i}\,\N^3_{\peso{0}{\lambda_2-2i}}
-
\sum_{i=1}^{m} q^{4i-3}\,\N^2_{\peso{1}{\lambda_2-2i+1}},
\end{equation}
where Proposition~\ref{prop: inverse decomp} has been used to rewrite some differences of two $\N^3$–terms as a $\N^2$–term.

We now apply Lemma~\ref{lem: N6 in N5 and N3 in N2} to rewrite each $\N^3$ in terms of $\N^2$. 
After a reordering of the terms, we obtain 
\begin{equation} 
\begin{array}{rl} \N_{\lambda}^5 &= \displaystyle\sum_{i=0}^m \sum_{j=0}^{\lambda_2-2i} q^{4i+j} \N_{\peso{j}{\lambda_2-2i+j}}^2 - \sum_{i=1}^{m} q^{4i-3} \N_{\peso{1}{\lambda_2-2i+1}}^2 \\ & = \displaystyle \sum_{i=0}^m q^{4i}\N_{\peso{0}{\lambda_2,2i}}^2 +\sum_{i=1}^m\sum_{j=1}^{\lambda_2 -2i+1} q^{4i-2} \N_{\peso{j+1}{\lambda_2-2i-j+1}}^2. 
\end{array} 
\end{equation}
which is manifestly an atomic decomposition.

\smallskip

We now treat the case $\lambda_1\ge3$ with $\lambda_1\equiv 0 \pmod{3}$. 
Suppose inductively that $\N^5_{\peso{\mu_1}{\mu_2}}$ has an atomic decomposition for all $\mu_1<\lambda_1$.  
By Proposition~\ref{prop: inverse decomp},
\begin{equation}\label{eq:atomic6}
\N^5_{\lambda}
=
\N^4_{\lambda}
+
q\,\N^5_{\peso{\lambda_1-3}{\lambda_2+1}}.
\end{equation}
The element $\N^4_{\lambda}$ has an atomic decomposition by Lemma~\ref{lem: N6 in N5 and N3 in N2} together with Lemma~\ref{lemma N4 enN3}, and the second term in \eqref{eq:atomic6} does so by the inductive hypothesis. 
Hence $\N^5_{\lambda}$ has an atomic decomposition as well, completing the proof.
\end{proof}

\section{Second approach}

In this section we provide an alternative proof of the atomicity of Kazhdan--Lusztig basis elements. 
As already noted, the step-by-step decompositions for the pre-canonical bases are not necessarily positive in general. 
We show that a slight modification of the definition of the pre-canonical bases restores positivity at each layer. 


\begin{definition} \label{def:adjusted pre-can}
    Let $\gamma_i$ be the unique root of height $i$ for $i\in[2,5]$ and
   \begin{equation}
       \begin{array}{l}
 X_5=\{\peso{\lambda_1}{\lambda_2} \mid \lambda_2\ge 1\}\\[2pt]
 X_4=\{\peso{\lambda_1}{\lambda_2}\mid \lambda_1\geq 3\}\\[2pt]
 X_3=\{\peso{\lambda_1}{\lambda_2}\mid \lambda_1\geq 2\}\\[2pt]
 X_2=\{\peso{\lambda_1}{\lambda_2}\mid \lambda_1\geq 2, \lambda_2\geq 1\}.
       \end{array}
   \end{equation}

    We then define the \newword{adjusted pre-canonical bases} inductively by setting
    $\widetilde{\N}^6_{\lambda}={\N}^6_{\lambda}=\HH{\lambda}$ and for $k\in [2,5]$ and $\lambda\in X^+$ we set
    \[\widetilde{\N}^k_{\lambda}=\begin{cases}
    \widetilde{\N}^{k+1}_{\lambda}- q \widetilde{\N}^{k+1}_{\lambda-\gamma_k}, & \text{if }\lambda \in X_k;\\
    \widetilde{\N}^{k+1}_{\lambda}, & \text{otherwise.}\\
\end{cases}\]
\end{definition}

\begin{remark}
Note that $\lambda\in X_k$ implies that $\lambda-\gamma_k\in X^+$. Thus, the adjusted bases are well-defined. 
\end{remark}

For the adapted pre-canonical bases we  immediately have
\[\widetilde{\N}^{k+1}_{\lambda}\in  \sum_{\mu \leq \lambda} \mathbb{N}[q]  \widetilde{\N}^{k}_{\mu}. \]
It follows that $\HH{\lambda} =\widetilde{\N}^6_{\lambda} $ is positive when written in terms of $\{ \widetilde{\N}_{\mu}^2 \mid \mu \in X^+ \}$. 
Thus, in order to obtain the atomicity of $\HH{\lambda}$, we must compute $\widetilde{\N}^{2}_{\lambda}$.

We now extend the definition of the sets $X_i$ in Definition \ref{def:adjusted pre-can} for subsets $I$ of $[2,5]$.

\begin{definition}
    Let $I\subset [2,5]$. 
    If $I=\emptyset $ then we define $X_{I}=X^+$. 
    If $I=\{i\}$ then we set $X_I=X_{i}$. 
    Finally, if $|I|>1$ we define 
    $$X_I=\{\lambda\in X_{i_0}\mid \lambda-\gamma_{i_0}\in X_{I\setminus\{{i_0}\}}\}, $$
    where $i_0=\min I$.
\end{definition}

In \Cref{tab:XI} we list the explicit conditions for a weight 
$\lambda = \peso{\lambda_1}{\lambda_2}$ to belong to a given set $X_I$.

\begin{table}[ht]
  \centering
  \renewcommand{\arraystretch}{1.2}
  \begin{tabular}{|c|c|l|l|l|}
    \hline
    $|I|$ & $I$ & $\lambda-\Gamma_{I}$
          &  $\lambda\in X_I$
          &  $\lambda-\Gamma_{I}\in X^+$  \\
    \hline \hline
    0 & $\emptyset $ & $(\lambda_1,\lambda_2)_\varpi$
      &
      & \\ 
    \hline\hline
    \multirow{4}{*}{1}
      & $\{2\}$ & $(\lambda_1+1,\lambda_2-1)_\varpi$
      & if $\lambda_1>1$ and $\lambda_2>0$
      & if $\lambda_2>0$ \\\cline{2-5}
      & $\{3\}$ & $(\lambda_1-1,\lambda_2)_\varpi$
      & if $\lambda_1>1$
      & if $\lambda_1>0$ \\\cline{2-5}
      & $\{4\}$ & $(\lambda_1-3,\lambda_2+1)_\varpi$
      & if $\lambda_1>2$
      & if $\lambda_1>2$ \\\cline{2-5}
      & $\{5\}$ & $(\lambda_1,\lambda_2-1)_\varpi$
      & if $\lambda_2>0$
      & if $\lambda_2>0$ \\
    \hline\hline
    \multirow{6}{*}{2}
      & $\{2,3\}$ & $(\lambda_1,\lambda_2-1)_\varpi$
      & if $\lambda_1>1$ and $\lambda_2>0$
      & if $\lambda_2>0$ \\\cline{2-5}
      & $\{2,4\}$ & $(\lambda_1-2,\lambda_2)_\varpi$
      & if $\lambda_1>1$ and $\lambda_2>0$
      & if $\lambda_1>1$ \\\cline{2-5}
      & $\{2,5\}$ & $(\lambda_1+1,\lambda_2-2)_\varpi$
      & if $\lambda_1>1$ and $\lambda_2>1$
      & if $\lambda_2>1$ \\\cline{2-5}
      & $\{3,4\}$ & $(\lambda_1-4,\lambda_2+1)_\varpi$
      & if $\lambda_1>3$
      & if $\lambda_1>3$ \\\cline{2-5}
      & $\{3,5\}$ & $(\lambda_1-1,\lambda_2-1)_\varpi$
      & if $\lambda_1>1$ and $\lambda_2>0$
      & if $\lambda_1>0$ and $\lambda_2>0$ \\\cline{2-5}
      & $\{4,5\}$ & $(\lambda_1-3,\lambda_2)_\varpi$
      & if $\lambda_1>2$
      & if $\lambda_1>2$ \\
    \hline\hline
    \multirow{4}{*}{3}
      & $\{2,3,4\}$ & $(\lambda_1-3,\lambda_2)_\varpi$
      & if $\lambda_1>2$ and $\lambda_2>0$
      & if $\lambda_1>2$ \\\cline{2-5}
      & $\{2,3,5\}$ & $(\lambda_1,\lambda_2-2)_\varpi$
      & if $\lambda_1>1$ and $\lambda_2>1$
      & if $\lambda_2>1$ \\\cline{2-5}
      & $\{2,4,5\}$ & $(\lambda_1-2,\lambda_2-1)_\varpi$
      & if $\lambda_1>1$ and $\lambda_2>0$
      & if $\lambda_1>1$ and $\lambda_2>0$ \\\cline{2-5}
      & $\{3,4,5\}$ & $(\lambda_1-4,\lambda_2)_\varpi$
      & if $\lambda_1>3$
      & if $\lambda_1>3$ \\
    \hline\hline
    4 & $\{2,3,4,5\}$
      & $(\lambda_1-3,\lambda_2-1)_\varpi$
      & if $\lambda_1>2$ and $\lambda_2>0$
      & if $\lambda_1>2$ and $\lambda_2>0$ \\
    \hline
  \end{tabular}
  \caption{Conditions for $\lambda\in X_I$ and for $\lambda-\Gamma_I \in X^+$.}
  \label{tab:XI}
\end{table}

For a dominant weight $\lambda$ and $k\in [2,6]$ we define
\begin{equation}
U_\lambda^k = \{ I\subset [2,5] \mid \lambda \in X_I \mbox{ and } k\leq \min I  \},
\end{equation}
where we make the convention that $\min \emptyset \coloneqq \infty$. 
Furthermore, for $I\subset [2,5]$ we set $\Gamma_I= \sum_{i\in I}\gamma_i$.

With all these definitions at hand, we are in position to express the elements of the adjusted pre-canonical bases in terms of the canonical basis. 

\begin{lemma} \label{lem:adjusted-in-can}
    Let $\lambda \in X^+$ and $k\in [2,6]$. 
    Then, we have 
        \begin{equation}\label{eq:adjusted-in-can}
         \widetilde{\N}^{k}_{\lambda}=\sum_{I\in U_\lambda^k} (-q)^{|I|}\HH{\lambda-\Gamma_I}.
    \end{equation}
\end{lemma}
\begin{proof}
    We proceed by downward induction. 
    The base case $k=6$ is clear since $U_\lambda^6=\{\emptyset\}$. 
    Therefore, $\widetilde{\N}^{6}_{\lambda}=\HH{\lambda}$ as predicted by  \eqref{eq:adjusted-in-can}.
    We now fix $2\leq k \leq 5$ and assume that \eqref{eq:adjusted-in-can} holds for $k+1$. 
    We have
    \begin{equation}
        \begin{array}{rl}
            \widetilde{\N}_\lambda^k & \displaystyle = \sum_{\substack{ J\subset \{k\}\\\lambda \in X_J  }} (-q)^{|J|} \widetilde{\N}^{k+1}_{\lambda -\Gamma_J}   \\[10pt]
            &   \\
             & \displaystyle = \sum_{\substack{ J\subset \{k\}\\\lambda \in X_J  }} (-q)^{|J|} \sum_{I\in U_{\lambda-\Gamma_J}^{k+1}} (-q)^{|I|} \HH{\lambda-\Gamma_J-\Gamma_I}  \\[10pt]
             & \\
              & \displaystyle = \sum_{\substack{ J\subset \{k\}\\\lambda \in X_J  }}\sum_{I\in U_{\lambda-\Gamma_J}^{k+1}}  (-q)^{|J|+|I|}  \HH{\lambda-\Gamma_{J\cup I}}  \\[10pt]
             &  \\
             & \displaystyle =\sum_{K\in U_\lambda^k} (-q)^{|K|}\HH{\lambda-\Gamma_K},
        \end{array}
    \end{equation}
    as we wanted to show. 
\end{proof}

\begin{cor} \label{cor:adjusted-N2}
    Let $\lambda \in X^+$. 
    Then, we have 
    \begin{equation} \label{eq:adjustedN2}
        \widetilde{\N}_{\lambda}^2 = \sum_{\substack{I\subset [2,5]\\ \lambda \in X_I}}  (-q)^{|I|} \HH{\lambda - \Gamma_I}.
    \end{equation}
\end{cor}

\begin{proof}
    This follows from a direct application of Lemma \ref{lem:adjusted-in-can} with $k=2$, once we observe that $U_{\lambda}^2=\{ I\subset [2,5] \mid \lambda \in X_I \}$.
\end{proof}

\begin{prop}\label{prop: N2s}
Let $\lambda = \peso{\lambda_1}{\lambda_2} \in X^+$. 
Then we have
\begin{equation} \label{eq:N2 adapted-vs-precan}
  \widetilde{\N}^{2}_{\lambda}-{\N}^{2}_{\lambda}   = \begin{cases}
      0,  & \mbox{if } \lambda_1\geq 3 \mbox{ or } \lambda_1+\lambda_2 <2;\\[5pt]
      q^2\widetilde{\N}^{2}_{\peso{0}{\lambda_2}}, & \mbox{if  } \lambda_1=2 \mbox{ and } \lambda_2\geq 0  ;  \\[5pt]
\displaystyle      q^2\widetilde{\N}^2_{\peso{1}{\lambda_2-1}} + \sum_{k=1}^{\lambda_2} q^k {\N}^2_{\peso{1+k}{\lambda_2-k}}, & \mbox{if  } \lambda_1=1 \mbox{ and } \lambda_2\geq 1  ;  \\[5pt]
q^4\widetilde{\N}^{2}_{\peso{0}{\lambda_2-2}} + \displaystyle \sum_{k=2}^{\lambda_2}q^k {\N}^2_{\peso{k}{\lambda_2-k}},& \mbox{if  } \lambda_1=0 \mbox{ and } \lambda_2\geq 2  .  
  \end{cases}
\end{equation}
\end{prop}

\begin{proof}
We first recall that using the notation in this section we have
\begin{equation}
    \N_{\lambda}^2 = \sum_{I\subset [2,5]} (-q)^{|I|} \tilde{\HH{}}_{\lambda-\Gamma_I}.
\end{equation}
Thus  Corollary \ref{cor:adjusted-N2} yields
\begin{equation} \label{eq:difN2}
    \widetilde{\N}^{2}_{\lambda}- {\N}^2_{\lambda} = \sum_{\substack{I\subset [2,5]\\ \lambda \in X_I}}  (-q)^{|I|} \HH{\lambda - \Gamma_I} - 
    \sum_{I\subset [2,5]} (-q)^{|I|} \tilde{\HH{}}_{\lambda-\Gamma_I}.
\end{equation}

We first focus in the first row of \eqref{eq:N2 adapted-vs-precan}. 
Let us suppose that $\lambda_1>3$ and $\lambda_2>1$.
We look at  \Cref{tab:XI}.
On the one hand, we have that $\lambda \in X_I$ for all $I\subset [2,5]$. 
On the other hand, we have $\lambda -\Gamma_I \in X^{+}$ for all $I\subset [2,5]$. 
Therefore, $\tilde{\HH{}}_{\lambda-\Gamma_I} = \HH{\lambda-\Gamma_I}$. 
Putting these two facts together we see that the right-hand side of \eqref{eq:difN2} vanishes. 
We conclude that  $\widetilde{\N}^{2}_{\lambda}= {\N}^2_{\lambda}$ in this case. 

For the next cases we need the following.

\begin{claim}\label{claim: -1-is-singular}
For all $k,n\in \Z$ we have
\begin{equation} 
    \begin{array}{l}
       s_1\cdot\peso{-k-1}{n}   = \peso{-k-1}{n} + k\alpha_1 = \peso{k-1}{n-k},  \\[3pt]
       s_2\cdot\peso{n}{-k-1}   =\peso{n}{-k-1} + k\alpha_2 = \peso{n-3k}{k-1}.
    \end{array}
\end{equation}
In particular,  the weights $\peso{-1}{\lambda_2}$ and $\peso{\lambda_1}{-1}$ are singular for all $\lambda_1,\lambda_2\in \Z$.
\end{claim}
\begin{proof}
   This follows directly by applying the definitions and using  $\rho=\peso{1}{1}$. 
\end{proof}
We now assume that $\lambda_{1}=3$ and $\lambda_{2}>1$. 
 \Cref{tab:XI} implies that $\lambda \in X_I$ and $\lambda -\Gamma_I \in X^{+}$ for all $I\subset [2,5]$ such that $I\neq \{3,4\}$ and $I\neq \{3,4,5\}$. 
By \eqref{eq:difN2} we obtain
\begin{equation}
    \widetilde{\N}^{2}_{\lambda}- {\N}^2_{\lambda} = -q^2 \tilde{\HH{}}_{\peso{-1}{\lambda_2+1}}   +q^{3}\tilde{\HH{}}_{\peso{-1}{\lambda_2}}=0, 
\end{equation}
where the last equality follows from Claim \ref{claim: -1-is-singular}.

The case  $\lambda_{1}>3$ and $\lambda_{2}=1$ is dealt with similarity. Indeed, there are only two sets that are neglected in \eqref{eq:adjustedN2}: $I=\{2,5 \}$ and $I=\{2,3,5 \}$. 
Furthermore, by Claim \ref{claim: -1-is-singular} the corresponding weights $\lambda-\Gamma_I$ are singular. 
This shows $\widetilde{\N}^{2}_{\lambda}- {\N}^2_{\lambda}=0$ in this case. 

We now suppose that $\lambda_1>3$ and $\lambda_2=0$. 
By combining \Cref{tab:XI} and Claim \ref{claim: -1-is-singular}, and arguing as in the previous paragraphs,  we obtain 

\begin{align*}
        \widetilde{\N}^{2}_{\peso{\lambda_1}{0}}-{\N}^{2}_{\peso{\lambda_1}{0}}
        &= -q^2\tilde{\HH{}}_{\peso{\lambda_1-2}{0}}-q^2\tilde{\HH{}}_{\peso{\lambda_1+1}{-2}}+q^3\tilde{\HH{}}_{\peso{\lambda_1-3}{0}}+q^3\tilde{\HH{}}_{\peso{\lambda_1}{-2}} \\
        & = -q^2\tilde{\HH{}}_{\peso{\lambda_1-2}{0}}+q^2\tilde{\HH{}}_{\peso{\lambda_1-2}{0}}+q^3\tilde{\HH{}}_{\peso{\lambda_1-3}{0}}-q^3\tilde{\HH{}}_{\peso{\lambda_1-3}{0}} \\
        & =0.
\end{align*}

Finally,  for $\lambda \in \{ \peso{0}{0}, \peso{0}{1}, \peso{1}{0}, \peso{3}{0}, \peso{3}{1}  \}$, a direct computation shows the equality $\widetilde{\N}^{2}_{\lambda}- {\N}^2_{\lambda}=0$.
This completes the proof of the first row of \eqref{eq:N2 adapted-vs-precan}.\\

We now move to the second row of \eqref{eq:N2 adapted-vs-precan}. 
We assume that $\lambda_1=2$ and $\lambda_2\ge1$. 
Our starting point is \eqref{eq:difN2}. 
Using \Cref{tab:XI}, we expand 
$\widetilde{\N}^{2}_{\peso{2}{\lambda_2}} - {\N}^{2}_{\peso{2}{\lambda_2}}$. 
Next, we apply Claim \ref{claim: -1-is-singular} to discard the $\tilde{\HH{}}$--terms indexed by singular weights. 
Combining these steps, we obtain
\begin{equation}
\begin{array}{rl}
    \widetilde{\N}^{2}_{\peso{2}{\lambda_2}}-{\N}^{2}_{\peso{2}{\lambda_2}}&= -q^2\tilde{\HH{}}_{\peso{-2}{\lambda_2+1}}+q^3\tilde{\HH{}}_{\peso{-2}{\lambda_2}}\\[5pt]
        &= q^2{\HH{}}_{\peso{0}{\lambda_2}}-q^3{\HH{}}_{\peso{0}{\lambda_2-1}}\\[5pt]
        &=q^2\widetilde{\N}^5_{\peso{0}{\lambda_2}}\\[5pt]
        &=q^2\widetilde{\N}^2_{\peso{0}{\lambda_2}}.
\end{array}
\end{equation}
On the other hand, a direct computations shows that $\widetilde{\N}^{2}_{\peso{2}{0}}-{\N}^{2}_{\peso{2}{0}}={\N}^{2}_{\peso{2}{0}}$. 
This finishes the proof of the second row in \eqref{eq:N2 adapted-vs-precan}.

We now focus on the third row of \eqref{eq:N2 adapted-vs-precan}. Assume that $\lambda_1=1$ and $\lambda_2\ge 1$. 
As in the previous cases, we combine \Cref{tab:XI} and Claim \ref{claim: -1-is-singular} to obtain 

\begin{equation}
\begin{array}{rl}
    \widetilde{\N}^{2}_{\peso{1}{\lambda_2}}-{\N}^{2}_{\peso{1}{\lambda_2}} =&q\tilde{\HH{}}_{\peso{2}{\lambda_2-1}}
\,    \,\,+ \pairA{q\tilde{\HH{}}_{\peso{0}{\lambda_2}}}
 \qquad   + \pairA{q\tilde{\HH{}}_{\peso{-2}{\lambda_2+1}}} 
    - \, \pairB{q^2\tilde{\HH{}}_{\peso{1}{\lambda_2-1}}}-\\[10pt]
        &q^2\tilde{\HH{}}_{\peso{2}{\lambda_2-2}}
    - \pairB{q^2\tilde{\HH{}}_{\peso{-3}{\lambda_2+1}}}
    - \pairC{ q^2\tilde{\HH{}}_{\peso{0}{\lambda_2-1}}}
   \, - \pairC{ q^2\tilde{\HH{}}_{\peso{-2}{\lambda_2}}} \,\,\,+\\[10pt]
        &q^3\tilde{\HH{}}_{\peso{-2}{\lambda_2}}
  \,\,\,  + \pairD{q^3\tilde{\HH{}}_{\peso{1}{\lambda_2-2}}}
 \,\,\,   + \pairD{q^3\tilde{\HH{}}_{\peso{-3}{\lambda_2}}}
  \,\,\,      - \,\, q^4\tilde{\HH{}}_{\peso{-2}{\lambda_2-1}}\\[10pt]
= &  q \widetilde{\N}^{5}_{\peso{2}{\lambda_2-1}} - q^3 \widetilde{\N}^{5}_{\peso{0}{\lambda_2-1}} \\ [10pt]
= &  q \widetilde{\N}^{3}_{\peso{2}{\lambda_2-1}}+q^2 \widetilde{\N}^{3}_{\peso{1}{\lambda_2-1}}- q^3 \widetilde{\N}^{3}_{\peso{0}{\lambda_2-1}} \\ [10pt]
= & \displaystyle  \sum_{k=0}^{\lambda_2-1} q^{k+1}  \widetilde{\N}^{2}_{\peso{2+k}{\lambda_2-1-k}}
+q^2 \widetilde{\N}^{2}_{\peso{1}{\lambda_2-1}}- q^3 \widetilde{\N}^{2}_{\peso{0}{\lambda_2-1}} \\ [10pt]
= & \displaystyle q  \widetilde{\N}^{2}_{\peso{2}{\lambda_2-1}}+  \sum_{k=2}^{\lambda_2} q^{k}  \widetilde{\N}^{2}_{\peso{1+k}{\lambda_2-k}}
+q^2 \widetilde{\N}^{2}_{\peso{1}{\lambda_2-1}}- q^3 \widetilde{\N}^{2}_{\peso{0}{\lambda_2-1}} \\ [10pt]
= & \displaystyle q  \N^{2}_{\peso{2}{\lambda_2-1}} + \sum_{k=2}^{\lambda_2} q^{k}  {\N}^{2}_{\peso{1+k}{\lambda_2-k}}+q^2 \widetilde{\N}^{2}_{\peso{1}{\lambda_2-1}} \\ [10pt]
= & \displaystyle  \sum_{k=1}^{\lambda_2} q^{k}  \N^{2}_{\peso{1+k}{\lambda_2-k}}
+q^2 \widetilde{\N}^{2}_{\peso{1}{\lambda_2-1}},  \\ [10pt]
\end{array}
\end{equation}
which matches with the equality in \eqref{eq:N2 adapted-vs-precan}.
We stress that to obtain the above equalities we have also used the definition of the adjusted pre-canonical bases elements and the (already proven) first and second row of \eqref{eq:N2 adapted-vs-precan}.

We proceed in a similar fashion for proving the last row in \eqref{eq:N2 adapted-vs-precan}. Indeed, after using \Cref{tab:XI} and Claim \ref{claim: -1-is-singular} we obtain for $\lambda_2\ge 2$ that
\begin{equation}
\begin{array}{rl}
 \widetilde{\N}^{2}_{\peso{0}{\lambda_2}}-{\N}^{2}_{\peso{0}{\lambda_2}} =     &   q^2 \HH{\peso{2}{\lambda_2-2}}-q^3 \tilde{\HH{}}_{\peso{2}{\lambda_2-3}}
     -q^3 \HH{\peso{1}{\lambda_2-2}}+q^4 \tilde{\HH{}}_{\peso{1}{\lambda_2-3}}\\[5pt]
     =& q^2 \widetilde{\N}^{4}_{\peso{2}{\lambda_2-2}} -q^3\widetilde{\N}^{4}_{\peso{2}{\lambda_2-3}} \\[5pt]
     =& q^2 \widetilde{\N}^{3}_{\peso{2}{\lambda_2-2}}  \\[5pt]
     =&   q^4\widetilde{\N}^{2}_{\peso{0}{\lambda_2-2}} + \displaystyle \sum_{k=2}^{\lambda_2}q^k {\N}^2_{\peso{k}{\lambda_2-k}}.
\end{array}
\end{equation}
as we wanted to show. 
\end{proof}

We then obtain the existence of an atomic decomposition for all $\HH{\lambda}$   with $\lambda\in X^+$ as a corollary.

\begin{proof}[Second proof of \Cref{thr:atomic}] 
By definition of the $\widetilde{\N}$-elements we have that 
\begin{equation}
    \HH{\lambda} =  \widetilde{\N}^6_{\lambda}\in   \sum_{\mu \leq \lambda} \mathbb{N}[q]\widetilde{\N}^2_{\mu}. 
\end{equation}

Therefore, to prove that $\HH{\lambda}$ has an atomic decomposition it is enough to show that the elements $\widetilde{\N}^2_{\mu}$ have an atomic decomposition. 
We demonstrate this by induction on the dominance order.

If  $\lambda =\peso{0}{0}$ Proposition \ref{prop: N2s} yields $\widetilde{\N}^2_{\peso{0}{0}}=\N^2_{\peso{0}{0}} $.
This is the base of our induction. 
We now fix $ \lambda >0$ and assume  that $\widetilde{\N}^2_\mu$ has an atomic decomposition for all $\mu <\lambda$.
Proposition \ref{prop: N2s} implies that 
\begin{equation}
   \widetilde{\N}^2_{\lambda} \in \sum_{\mu \leq \lambda} \mathbb{N}[q]{\N}^2_{\mu}+\sum_{\mu < \lambda} \mathbb{N}[q]\widetilde{\N}^2_{\mu}.
\end{equation}
Using our induction hypothesis we get that $  \widetilde{\N}^2_{\lambda}$ has an atomic decomposition. 
\end{proof}

\section{Computing Kostka--Foulkes polynomials}

In this section we explain how the atomic decomposition can be used to compute
Kostka--Foulkes polynomials.

We begin by observing that the two approaches developed in the previous
sections not only establish the atomic decomposition for canonical basis
elements, but in fact yield explicit algorithms for computing these
decompositions.

We now introduce the standard basis of the spherical Hecke algebra
\(\mathcal{H}^{\mathbf{sph}}\).
Unlike the canonical or atomic bases, the standard basis is not simply the set
\(\{\mathbf{H}_{\theta(\lambda)} \mid \lambda \in X^{+}\}\), since these elements do not lie in \(\mathcal{H}^{\mathbf{sph}}\).
To obtain the correct definition, one must sum over the standard basis
elements of \(\mathcal{H}\) that lie in the double coset
\(W_f \theta(\lambda) W_f\), weighted by the length. Explicitly,
\[
  \mathbf{H}_{\lambda}
  \coloneqq
  \sum_{w \in W_f \theta(\lambda) W_f}
  v^{\ell(\theta(\lambda)) - \ell(w)}\,
  \mathbf{H}_{w}.
\]
With this definition, the standard basis of \(\mathcal{H}^{\mathbf{sph}}\) is
precisely \(\{\mathbf{H}_{\lambda} \mid \lambda \in X^{+}\}\).

It is straightforward to verify that
\[
  \N_{\lambda}^{2}
  \;=\;
  \N_{\lambda}
  \;=\;
  \sum_{\mu \leq \lambda}
    q^{\operatorname{ht}(\lambda-\mu)}\,\mathbf{H}_{\mu}.
\]
Moreover, if we expand the canonical basis element as
\[
  \HH{\lambda}
  \;=\;
  \sum_{\mu \leq \lambda}
    K_{\lambda,\mu}(q)\,\mathbf{H}_{\mu},
\]
then
\[
  K_{\lambda,\mu}(q)
  \;=\;
  h_{\theta(\mu),\,\theta(\lambda)}(v),
  \qquad\text{with } q = v^{2}.
\]
In other words, the polynomials \(K_{\lambda,\mu}(q)\) are precisely the
Kostka--Foulkes polynomials.

\medskip

We now summarize the procedure for computing a Kostka--Foulkes polynomial
\(K_{\lambda,\mu}(q)\) from the atomic decomposition.

\begin{enumerate}
  \item Compute the atomic expansion of the canonical basis element:
  \[
    \HH{\lambda}
    =
    \sum_{\nu \leq \lambda}
      a_{\lambda,\nu}(q)\,\N_{\nu}.
  \]

  \item Determine the set of weights \(\nu\) appearing in this expansion such
  that \(\mu \leq \nu\).

  \item Finally,
  \[
    K_{\lambda,\mu}(q)
    =
    \sum_{\nu}
      q^{\operatorname{ht}(\nu-\mu)}\,
      a_{\lambda,\nu}(q),
  \]
  where the sum runs over all \(\nu\) satisfying \(\mu \leq \nu\) in step~(2).
\end{enumerate}

We implemented this algorithm in \textsc{SageMath}~10.6, and the corresponding
code can be found at (\href{https://github.com/bamuniz/AtomicDecompositionG2}{github.com/bamuniz/AtomicDecompositionG2}). 
To obtain the atomic decomposition we have used the formulas in \S\ref{sec:SBS-decompositions}.
Although the code is largely
self-explanatory, we conclude this section with an example illustrating its
usage.

\begin{example}
    Let $\lambda =\peso{a}{b} \in X^+$. If we want to compute the atomic decomposition of $\HH{\lambda}$ then we use the command \textbf{atomic(a,b)}. For instance if $\lambda=\peso{2}{4}$ the program  should return

\[\scalemath{0.85}{
\begin{array}{rl}
&\N_{\peso{2}{4}} + q\,\N_{\peso{3}{3}} + q\,\N_{\peso{1}{4}}  + q^2\,\N_{\peso{4}{2}} + (q^2 + q)\,\N_{\peso{2}{3}} + q^3\,\N_{\peso{5}{1}} +\\[5pt]
& q^2\,\N_{\peso{0}{4}} + (q^3 + q^2)\,\N_{\peso{3}{2}} + q^4\,\N_{\peso{6}{0}}+ (q^3 + q^2)\,\N_{\peso{1}{3}} + (q^4 + q^3)\,\N_{\peso{4}{1}} 
        +  \\[5pt]
& (2q^4 + q^3 + q^2)\,\N_{\peso{2}{2}}+ (q^5 + q^4)\,\N_{\peso{5}{0}} + q^3\,\N_{\peso{0}{3}} + (2q^5 + q^4 + q^3)\,\N_{\peso{3}{1}}+\\[5pt]
\HH{\peso{2}{4}} =   &      (q^5 + q^4 + q^3)\,\N_{\peso{1}{2}} 
        + (2q^6 + q^5 + q^4)\,\N_{\peso{4}{0}} 
        + (q^6 + 2q^5 + q^4 + q^3)\,\N_{\peso{2}{1}} + \\[5pt]
&   (q^6 + q^4)\,\N_{\peso{0}{2}} + (q^7 + 2q^6 + q^5 + q^4)\,\N_{\peso{3}{0}} +(q^7 + q^6 + q^5 + q^4)\,\N_{\peso{1}{1}} + \\[5pt]
&   (2q^8 + q^7 + 2q^6 + q^5 + q^4)\,\N_{\peso{2}{0}}  + (q^9 + q^8 + q^7 + q^6 + q^5)\,\N_{\peso{1}{0}} + \\[5pt]
      &   (q^7 + q^5)\,\N_{\peso{0}{1}} +(q^{10} + q^8 + q^6)\,\N_{\peso{0}{0}}.\\[5pt] 
\end{array}}
\]

On the other hand, if we want to compute $K_{\lambda, \mu}(q)$ for $\lambda=\peso{a}{b}$ and $\mu=\peso{c}{d}$, then we use the command \textbf{KF((a,b)(c,d))}.
For instance, if $\lambda= \peso{6}{9}$ and $\mu=\peso{3}{2}$ then the program  should return

\[
\begin{array}{rl}
   & \quad q^{44} +\quad q^{43} + \,\,\,2q^{42} + \,\, \,3q^{41} +\,\, \,5q^{40} +\,\,\, 6q^{39} + \,\,\,9q^{38} + 10q^{37} + 14q^{36} +   \\[10pt]
      & 16q^{35}+ 19q^{34} + 21q^{33} + 26q^{32} + 27q^{31} + 31q^{30} + 33q^{29} + 37q^{28} + 38q^{27} + \\
 K_{\peso{6}{9},\peso{3}{2}} =& \\
     &  42q^{26} + 42q^{25}+  46q^{24} + 46q^{23} + 48q^{22} + 47q^{21} + 51q^{20} + 48q^{19} + 50q^{18} +  \\[10pt]
     & 45q^{17} + 40q^{16} + 31q^{15}+ 26q^{14} + 18q^{13} + 14q^{12} +\,\,\, 8q^{11} + \,\,\,4q^{10} + q^{9}.
\end{array}
\]
\end{example}

\printbibliography

\end{document}